\documentclass[11pt, a4paper]{amsart}

\usepackage[hmargin=30mm, vmargin=25mm, includefoot, twoside]{geometry}
\usepackage[bookmarksopen=true]{hyperref}
\usepackage[english]{babel}
\usepackage{amsmath,amssymb}
\usepackage{mathtools}
\usepackage{graphicx}
\usepackage[utf8]{inputenc}
\usepackage[figurename=Fig.]{caption}

\newtheorem{thmintro}{Theorem}

\newtheorem{corintro}[thmintro]{Corollary}
\newtheorem{theorem}{Theorem}[section]
\newtheorem{corollary}[theorem]{Corollary}
\newtheorem{lemma}[theorem]{Lemma}
\newtheorem{prop}[theorem]{Proposition}
\theoremstyle{definition}

\newtheorem{example}[theorem]{Example}
\newtheorem{definition}[theorem]{Definition}

\newcommand{\NN}{\mathbb{N}}
\newcommand{\RR}{\mathbb{R}}
\newcommand{\CC}{\mathbb{C}}
\newcommand{\ZZ}{\mathbb{Z}}
\newcommand{\HHH}{\mathcal{H}}
\newcommand{\ZZZ}{\mathcal{Z}}
\newcommand{\inv}{^{-1}}
\newcommand{\co}{\colon\thinspace}

\DeclareMathOperator{\Ch}{Ch}
\DeclareMathOperator{\dist}{d}
\DeclareMathOperator{\dc}{\dist_{\Ch}}
\DeclareMathOperator{\Int}{int}
\DeclareMathOperator{\Pc}{Pc}
\DeclareMathOperator{\proj}{proj}
\DeclareMathOperator{\Stab}{Stab}
\DeclareMathOperator{\supp}{supp}
\DeclareMathOperator{\Hom}{Hom}
\DeclareMathOperator{\End}{End}

\numberwithin{equation}{section}

\begin{document}

\renewcommand{\proofname}{{\bf Proof}}

\title[On the centre of Iwahori-Hecke algebras]{On the centre of Iwahori-Hecke algebras}

\author[T.~Marquis]{Timoth\'ee \textsc{Marquis}$^*$}
\address{UCLouvain, IRMP, 1348 Louvain-la-Neuve, Belgium}
\email{timothee.marquis@uclouvain.be}
\thanks{$^*$F.R.S.-FNRS Research Associate}

\author[S.~Raum]{Sven Raum$^\dagger$}
\address{Department of Mathematics, Stockholm University, SE-106 91 Stockholm, Sweden}
\email{raum@math.su.se}
\thanks{$^\dagger$ supported by the Swedish Research Council through grant number 2018-04243}


\dedicatory{In memory of Jacques Tits}

\begin{abstract}
  We prove triviality of the centre of arbitrary Hecke algebras of irreducible non-finite non-affine type.  This result is obtained as a consequence of the following structure result for conjugacy classes of the underlying Coxeter groups. If $W$ is any infinite irreducible Coxeter group and $w \in W$ is a nontrivial element that is assumed not be a translation in case $W$ is affine, then there is an infinite sequence of conjugates of $w$ by Coxeter generators whose length is non-decreasing and tends to infinity.
\end{abstract}



\maketitle


\section{Introduction}

Hecke algebras are deformations of group algebras associated with a Coxeter system and suitable deformation parameters.  Particular attention has been paid to the class of Hecke algebras associated to finite or affine Coxeter groups.  These efforts were motivated by representation theory of finite groups of Lie type and $p$-adic reductive groups.  In the spherical case, Tits' deformation theorem shows that the complex Hecke algebra does not depend on the deformation parameter \cite[Chapitre IV, Exercise 27]{Bourbaki}, and in particular provides a calculation of the centre's dimension, which equals the number of irreducible complex representations of the Coxeter group. In the affine case, a key structural theorem proved by Bernstein and presented in \cite{lus89} shows that the centre of an affine Hecke algebra is finitely generated and that the algebra itself is a finitely generated module over its centre.

Due to a lack of established methods, for example from algebraic geometry and the combinatorics of root systems, Hecke algebras associated with infinite non-affine Coxeter groups have received less attention than their counterparts arising from finite and reductive groups. In particular, it remained an open problem to determine the centre of Hecke algebras for such types. In this article, we completely clarify this question, proving triviality of the centre of Hecke algebras (in the sense that it equals the base ring) and hence --- combined with Tits' and Bernstein's theorems --- establish a dichotomy between finite and affine types on the one hand and arbitrary indefinite types (i.e. non-finite non-affine types) on the other hand. This result can be viewed as a vast generalisation of the classical fact, whose proof is based on the work of Jacques Tits, that infinite irreducible Coxeter groups have trivial centre \cite[Section V.4, Exercise 3]{Bourbaki}.

\begin{thmintro}\label{thmintro:Hecke-generic}
  Let $(W,S)$ be a Coxeter system of irreducible indefinite type, let $R$ be a ring and $(a_s,b_s)_{s \in S}$ a deformation parameter.  Then the centre of the generic Hecke algebra $\HHH(W, S, (a_s, b_s)_{s \in S})$ is trivial. 
\end{thmintro}

As an immediate corollary, we obtain the following result for the two most-studied classes of Iwahori-Hecke algebras.
\begin{corintro}\label{corintro:Hecke}
  Let $(W,S)$ be a Coxeter system of irreducible indefinite type.
  \begin{itemize}
  \item If $\mathbf{q} = (q_s)_{s \in S}$ is a deformation multiparameter of positive real numbers, then the centre of the Hecke algebra $\CC_{\mathbf{q}}[W] = \HHH(W, S, (q_s^{1/2} - q_s^{-1/2}, 1)_{s \in S})$ over $\CC$ is trivial.
  \item If $L: W \to \ZZ$ is a length function associated with $(W,S)$, then the centre of the Hecke algebra $\HHH(W, S, (v^{L(s)},1)_{s \in S})$ over $\ZZ[v, v^{-1}]$ is trivial.
  \end{itemize}
\end{corintro}

We refer the reader to Section~\ref{sec:center} for preliminaries on Hecke algebras.

Efforts to understand possible decompositions of Hecke algebras of arbitrary indefinite type and their associated von Neumann algebraic completions were started after the turn of the millenium, motivated by \mbox{L$^2$-cohomology} of buildings \cite{dymara06,davisdymarajanusykiewiczokun07}. For the class of irreducible indefinite right-angled Coxeter systems this led in \cite{garncarek16, raumskalski20} to a description of Hecke von Neumann algebras as a direct sum of a simple von Neumann algebra with finitely many copies of $\CC$ intersecting the Hecke algebra trivially. In particular, this covers the first part of Corollary~\ref{corintro:Hecke} in the special case of right-angled Coxeter systems. We consider the present work as an indication that a similar decomposition should exist for arbitrary Hecke von Neumann algebras of indefinite type.

The reason that we can obtain Theorem~\ref{thmintro:Hecke-generic} in the given generality is that we show a structural result for conjugacy classes of the Coxeter groups under consideration.  Indeed, it is a consequence of the following Coxeter group-theoretic result of independent interest.

Given $w,w'\in W$ and $s\in S$, write $w\stackrel{s}{\leftarrow}w'$ if $w'=sws$ and $\ell_S(w)\leq\ell_S(w')$.

\begin{thmintro}\label{thmintro:Cox}
Let $(W,S)$ be an infinite irreducible Coxeter system, and let $w\in W\setminus\{1\}$. If $W$ is of affine type, we moreover assume that $w$ is not a translation. Then there is a sequence $$w=w_0\stackrel{s_1}{\leftarrow}w_1\stackrel{s_2}{\leftarrow}w_2\stackrel{s_3}{\leftarrow}\dots\quad\textrm{for some $s_i\in S$}$$ such that $\ell_S(w_n)\stackrel{n\to\infty}{\to}\infty$.
\end{thmintro}

Theorem~\ref{thmintro:Cox} for $W$ of affine type is the main result of \cite{Ros16} (the paper \cite{Ros16} actually deals with the more general setting of Iwahori-Weyl groups), and our arguments provide an alternative shorter proof of that case. Another related paper is \cite{Badupward}, whose main result asserts that, if $w\in W$ satisfies $\ell_S(rwr)\leq \ell_S(w)$ for any reflection $r\in S^W$(which is much stronger than only requiring it for elements of $S$), then all conjugates of $w$ have the same length (and this can only occur if $W$ is of affine type and $w$ is a translation, see Lemma~\ref{lemma:ICC}).

Note that the presence of translations in the affine case is precisely the reason why the centre of the corresponding Hecke algebras is nontrivial, and why Theorem~\ref{thmintro:Cox} does not imply the conclusion of Theorem~\ref{thmintro:Hecke-generic} in that case.

The proof of Theorem~\ref{thmintro:Cox} is split in two cases: the case where $W$ has an infinite proper parabolic subgroup, which we explore in Section~\ref{section:CGWAIPPS}, and the complementary case (which includes the affine, as well as the so-called \emph{compact hyperbolic} types), which we explore in Section~\ref{section:CGWNIPPS}. The proof for the former case is more of combinatorial nature, using properties of double cosets of standard parabolic subgroups (as well as some properties of the Coxeter complex $\Sigma$ of $(W,S)$). The proof for the latter case is of geometric nature, using a CAT(0) metric realisation of $\Sigma$, called the \emph{Davis complex} of $(W,S)$. The necessary preliminaries for these proofs are exposed in Section~\ref{section:prelim}.


\section{Preliminaries on Coxeter groups and the Davis complex}\label{section:prelim}
The general reference for this section is \cite{BrownAbr} (see also \cite[\S 21]{MPW15} for \S\ref{subsection:prelim:res} and \cite{Moussong} for \S\ref{subsection:prelim:Davis}).


\subsection{Coxeter groups}
Let $(W,S)$ be a Coxeter system with finite generating set $S$. We denote by $\ell=\ell_S\co W\to \NN$ the word length with respect to $S$.

The subgroups $W_I:=\langle I\rangle\subseteq W$ ($I\subseteq S$) are called {\bf standard parabolic subgroups} of $W$, and their conjugates {\bf parabolic subgroups} of $W$. 

A subset $I\subseteq S$ (resp. $W_I$) is {\bf irreducible} if it does not decompose nontrivially as a disjoint union $I=I_1\sqcup I_2$ with $s_1s_2=s_2s_1$ for all $s_1\in I_1$ and $s_2\in I_2$. We call $I$ {\bf spherical} if $W_I$ is finite. We call $I$ of {\bf finite type} (resp. {\bf affine type}) if it is irreducible and $W_I$ is finite (resp. virtually abelian).
  

\subsection{Coxeter complexes}\label{subsection:prelim:CC}
The {\bf Coxeter complex} $\Sigma=\Sigma(W,S)$ of $(W,S)$ is the simplicial complex with simplices the cosets $wW_I$ ($w\in W$, $I\subseteq S$) and face relation $\leq$ the opposite of the inclusion relation. In particular, the maximal simplices, called {\bf chambers}, correspond to cosets of the form $wW_{\varnothing}=\{w\}$, that is, to elements of $W$. The chamber $\{1_W\}$ is called the {\bf fundamental chamber}, denoted $C_0$; the map $W\to\Ch(\Sigma):w\mapsto wC_0$ from $W$ to the set $\Ch(\Sigma)$ of chambers is then a bijection.

Two chambers are {\bf adjacent} if they are of the form $\{w\}$ and $\{ws\}$ for some $w\in W$ and $s\in S$ (in this case, they are called {\bf $s$-adjacent}). A {\bf gallery} in $\Sigma$ of {\bf length} $k\in\NN$ is a sequence $\Gamma=(D_0,D_1,\dots,D_k)$ such that $D_{i-1}$ is adjacent to $D_i$ (say $s_i$-adjacent) for each $i=1,\dots,k$. If $(s_1,\dots,s_k)\in J^k$ for some $J\subseteq S$, then $\Gamma$ is called a {\bf $J$-gallery}. The {\bf chamber distance} $\dc(C,D)$ between the chambers $C$ and $D$ is the length of a minimal-length gallery from $C$ to $D$. Alternatively, $\dc(vC_0,wC_0)=\ell(v\inv w)$ is the distance between $v$ and $w$ in the Cayley graph $\mathrm{Cay}(W,S)$ of $(W,S)$. A refinement of $\dc$ is the {\bf Weyl distance} $\delta\co \Ch(\Sigma)\times \Ch(\Sigma)\to W$ given by $\delta(vC_0,wC_0)=v\inv w$.

The group $W$ acts (by left translation) by simplicial isometries on $\Sigma$. The elements of $S$ are called {\bf simple reflections} and their conjugates {\bf reflections}. Each reflection $r\in W$ is uniquely determined by its fixed-point set $m$ in $\Sigma$ (called a {\bf wall}), and we then write $r=r_m$. Each wall $m$ determines two subcomplexes of $\Sigma$, called {\bf half-spaces}: if we again denote by $m$ the fixed-point set of $r_m$ in $\mathrm{Cay}(W,S)$, then $\mathrm{Cay}(W,S)\setminus m$ has two connected components, whose vertex sets are (under the identification $W\approx\Ch(\Sigma)$) the underlying chamber sets of these half-spaces. Two chambers $C,D$ are {\bf separated} by a wall $m$ if they lie in different half-spaces associated to $m$. Two adjacent chambers are separated by a unique wall: $wC_0$ and $wsC_0$ are separated by the wall $m$ with $r_m=wsw\inv$. The $|S|$ walls corresponding to $wsw\inv$, $s\in S$, are called the {\bf walls of $wC_0$}. The number of walls separating two chambers $C,D$ coincides with $\dc(C,D)$.


\subsection{Residues}\label{subsection:prelim:res}
Given a chamber $C\in\Ch(\Sigma)$ and a subset $J\subseteq S$, the set $R_J(C)$ of chambers connected to $C$ by a $J$-gallery is called a {\bf $J$-residue} (or {\bf residue of type} $J$, or just {\bf residue}). If $C=C_0$, we simply write $R_J:=R_J(C_0)$ and call it the {\bf standard $J$-residue}. Under the identification $W\approx\Ch(\Sigma)$, we have $R_J(wC_0)=wR_J\approx wW_J$; in particular, the stabilisers of residues are precisely the parabolic subgroups of $W$. A $J$-residue is {\bf spherical} if $J$ is spherical. A wall $m$ is called a {\bf wall of the residue $R$} if it separates two chambers of $R$. Alternatively, $m$ is a wall of $R$ if and only if $r_m\in\Stab_W(R):=\{w\in W \ | \ wR=R\}$.

Given a chamber $C$ and a residue $R$, there is a unique chamber of $R$ minimising the chamber distance from $C$ to chambers of $R$; it is denoted $\proj_R(C)$ and called the {\bf projection} of $C$ on $R$. It enjoys the following {\bf gate property}:
$$\dc(C,D)=\dc(C,\proj_R(C))+\dc(\proj_R(C),D) \quad\textrm{for all $D\in R$.}$$
Alternatively, $\proj_R(C)$ is the unique chamber $D$ of $R$ such that $C,D$ are not separated by any wall of $R$. The map $\proj_R\co\Ch(\Sigma)\to R$ does not increase the chamber distance, i.e. $$\dc(\proj_R(C),\proj_R(D))\leq\dc(C,D)\quad\textrm{for all $C,D\in\Ch(\Sigma)$.}$$
Note that $\proj_R$ is compatible with the $W$-action, in the sense that 
$$w\proj_R(C)=\proj_{wR}(wC)\quad\textrm{for all $w\in W$ and $C\in\Ch(\Sigma)$.}$$

Given two residues $R,R'$, the set $\proj_R(R')=\{\proj_R(C) \ | \ C\in R'\}$ is again a residue, called the {\bf projection} of $R$ on $R'$. If the restriction of $\proj_R$ to $R'$ is bijective, the residues $R,R'$ are called {\bf parallel} (and the inverse bijection is then given by the restriction of $\proj_{R'}$ to $R$). Alternatively, $R,R'$ are parallel if and only if they have the same set of walls if and only if $\Stab_W(R)=\Stab_W(R')$. If $R,R'$ are parallel, then the element $\delta(C,\proj_{R'}(C))\in W$ is independent of $C\in R$ and called the {\bf Weyl distance} from $R$ to $R'$, denoted $\delta(R,R')$ (note that $\delta(R',R)=\delta(R,R')\inv$).

We record the following result on parallel residues for future reference.

\begin{lemma}\label{lemma:MPWcriterionspherical}
Let $R,\overline{R}$ be parallel spherical residues of $\Sigma$, and let $x:=\delta(R,\overline{R})$. Suppose that $R$ is of type $J:=\{s\in S \ | \ \ell(sx)=\ell(x)+1\}$. Then $S$ is spherical.
\end{lemma}
\begin{proof}
This follows from \cite[Prop.~21.30]{MPW15} applied with $w:=x\inv$, $K:=J$, $T:=R$ and $R:=\overline{R}$.
\end{proof}


\subsection{Cosets of standard parabolic subgroups}
For $I,J\subseteq S$, set
$$\prescript{I}{}{W}=\{w\in W \ | \ \ell(sw)>\ell(w) \ \forall s\in I\}$$ 
and
$$W^J=\{w\in W \ | \ \ell(ws)>\ell(w) \ \forall s\in J\}.$$
Set also $\prescript{I}{}{W}^J:=\prescript{I}{}{W}\cap W^J$. Every coset $wW_I$ (resp. $W_Iw$) contains a unique element $w_0$ of minimal length, which is the unique element $w_0\in wW_I\cap W^I$ (resp. $w_0\in W_Iw\cap  \prescript{I}{}{W}$), and we have $\ell(w_0w_I)=\ell(w_0)+\ell(w_I)$ (resp. $\ell(w_Iw_0)=\ell(w_I)+\ell(w_0)$) for all $w_I\in W_I$. Geometrically, an element $w\in W$ belongs to $W^I$ if and only if $\proj_{R_I(wC_0)}(C_0)=wC_0$, so that the last equalities are reformulations of the gate property.

Similarly, every double coset $W_IwW_J$ contains a unique element of minimal length, or equivalently, a unique element in $\prescript{I}{}{W}^J$. The following lemma is well-known (we could not find its proof in the literature, so we include it here for the benefit of the reader).

\begin{lemma}\label{lemma:doubleparabdecomp}
Let $I,J\subseteq S$ and $w\in \prescript{I}{}{W}^J$. Set $H:=I\cap wJw\inv$. Then for any $x\in W_I$, we have $xw\in W^J$ if and only if $x\in W_I^H$. In particular, every element of $W_IwW_J$ can be uniquely written as $uwv$ with $u\in W_I^H$ and $v\in W_J$, and in that case $\ell(uwv)=\ell(u)+\ell(w)+\ell(v)$.
\end{lemma}
\begin{proof}
Let $x\in W_I$. If $s\in H$, then $xsw=xwt_s$ with $t_s:=w\inv sw\in J$, and $\ell(xsw)=\ell(xs)+\ell(w)$ as $w\in \prescript{I}{}{W}$. Hence if $xw\in W^J$, then $$\ell(x)+\ell(w)+1=\ell(xw)+1=\ell(xwt_s)=\ell(xsw)=\ell(xs)+\ell(w)\quad\textrm{for all $s\in H$},$$
so that $x\in W_I^H$. 

Conversely, let $x\in W_I^H$ and suppose for a contradiction that $\ell(xwt)<\ell(xw)$ for some $t\in J$. By \cite[Lemma~2.24]{BrownAbr}, we then have $xwt=x_1w$ for some $x_1\in W_I$, so that $wtw\inv\in W_I$. But then $\ell(w)+1=\ell(wt)=\ell(wtw\inv)+\ell(w)$, so that $\ell(wtw\inv)=1$ and hence $wtw\inv\in I$. Therefore, $s:=wtw\inv\in H$ and 
$$\ell(xs)=\ell(xsw)-\ell(w)=\ell(xwt)-\ell(w)<\ell(xw)-\ell(w)=\ell(x),$$
a contradiction.

For the uniqueness statement in the second part of the lemma, suppose that $u_1wv_1=u_2wv_2$ with $u_i\in W_I^H$ and $v_i\in W_J$. Then $$\ell(u_2)=\ell(u_2w)-\ell(w)=\ell(u_1wv_1v_2\inv)-\ell(w)=\ell(u_1)+\ell(v_1v_2\inv),$$ and similarly $\ell(u_1)=\ell(u_2)+\ell(v_2v_1\inv)=\ell(u_2)+\ell(v_1v_2\inv)$. Thus, $\ell(v_1v_2\inv)=0$, that is, $v_1=v_2$ and hence $u_1=u_2$.
\end{proof}


\subsection{Davis complex}\label{subsection:prelim:Davis}
The {\bf Davis complex} $X$ of $(W,S)$ is a metric realisation of $\Sigma$, turning $\Sigma$ into a complete CAT(0) cellular complex with piecewise Euclidean metric $\dist$ (in particular, it is an \emph{$M_0$-polyhedral complex of curvature $\leq 0$} in the sense of \cite{BHCAT0}), on which $W$ naturally acts by cellular isometries. It is constructed as follows.

If $I$ is a spherical subset of $S$, then $W_I$ acts as a linear reflection group on some $|I|$-dimensional Euclidean space $V_I$, with the simple reflections $s\in I$ acting as linear reflections. The walls $m_s$ fixed by $s\in I$ delimit a simplicial cone $C_I$ in $V_I$, and we let $x_I$ be the point in the interior of $C_I$ at distance $1$ from each of these walls. The convex hull of the orbit $W_Ix_I$ is then a convex polytope $P_I$, which we equip with the induced Euclidean metric from $V_I$ (the $1$-skeleton of this polytope is the Cayley graph of $(W_I,I)$).   

The construction of the cellular complex $X$ is now as follows. The $1$-skeleton $X^{(1)}$ of $X$ is the Cayley graph of $(W,S)$. Now, each spherical residue of $\Sigma$ (say of type $I$) corresponds to a subgraph of $X^{(1)}$ isomorphic to $\mathrm{Cay}(W_I,I)$ (recall that chambers of $\Sigma$ correspond to elements of $W$), and we glue to this subgraph the $|I|$-dimensional {\bf Coxeter cell} $P_I$ (more precisely, we perform this gluing process step by step, going from lower to top dimensional cells, gluing each time along lower dimensional cells, so that if $I\subseteq J$ are spherical, then $P_I$ will be a face of the polytope $P_J$). The metric $\dist$ on $X$ is then obtained by gluing together the Euclidean metrics on each of the cells $P_I$ ($I\subseteq S$ spherical).

Since $W$ acts on $\mathrm{Cay}(W,S)$ and preserves spherical residues, the $W$-action on $\Sigma$ indeed induces a $W$-action on $X$ by cellular isometries.

\begin{example}
If $(W,S)$ is infinite irreducible and all proper parabolic subgroups of $W$ are finite (i.e. all proper subsets of $S$ are spherical), then $X$ is just the standard geometric realisation of $\Sigma$ (minus the empty simplex): it is either a tessellation of the Euclidean space by congruent simplices (if $W$ is of affine type --- see Figure~\ref{figure:A2tilde} for an example), or else a tessellation of the hyperbolic space by congruent (compact) simplices (in that case, $W$ is said to be of {\bf compact hyperbolic} type --- see Figure~\ref{figure:hyperbolic} for an example). In the first case, $\dist$ is the usual Euclidean metric, whereas in the latter case, $\dist$ is quasi-isometric to the usual hyperbolic metric (which is of course not locally Euclidean). In both examples, the maximal Coxeter cells are the convex regular $2n$-gons (with $n=3$ on Figure~\ref{figure:A2tilde} and $n=4$ on Figure~\ref{figure:hyperbolic}) whose vertices are the barycentres of the triangles containing a given vertex of the tesselation.
\end{example}

\begin{figure}
\centering
\begin{minipage}[t]{.5\textwidth}
   \centering
  \includegraphics[trim = 13mm 5mm 21mm 4mm, clip, width= 7.5cm]{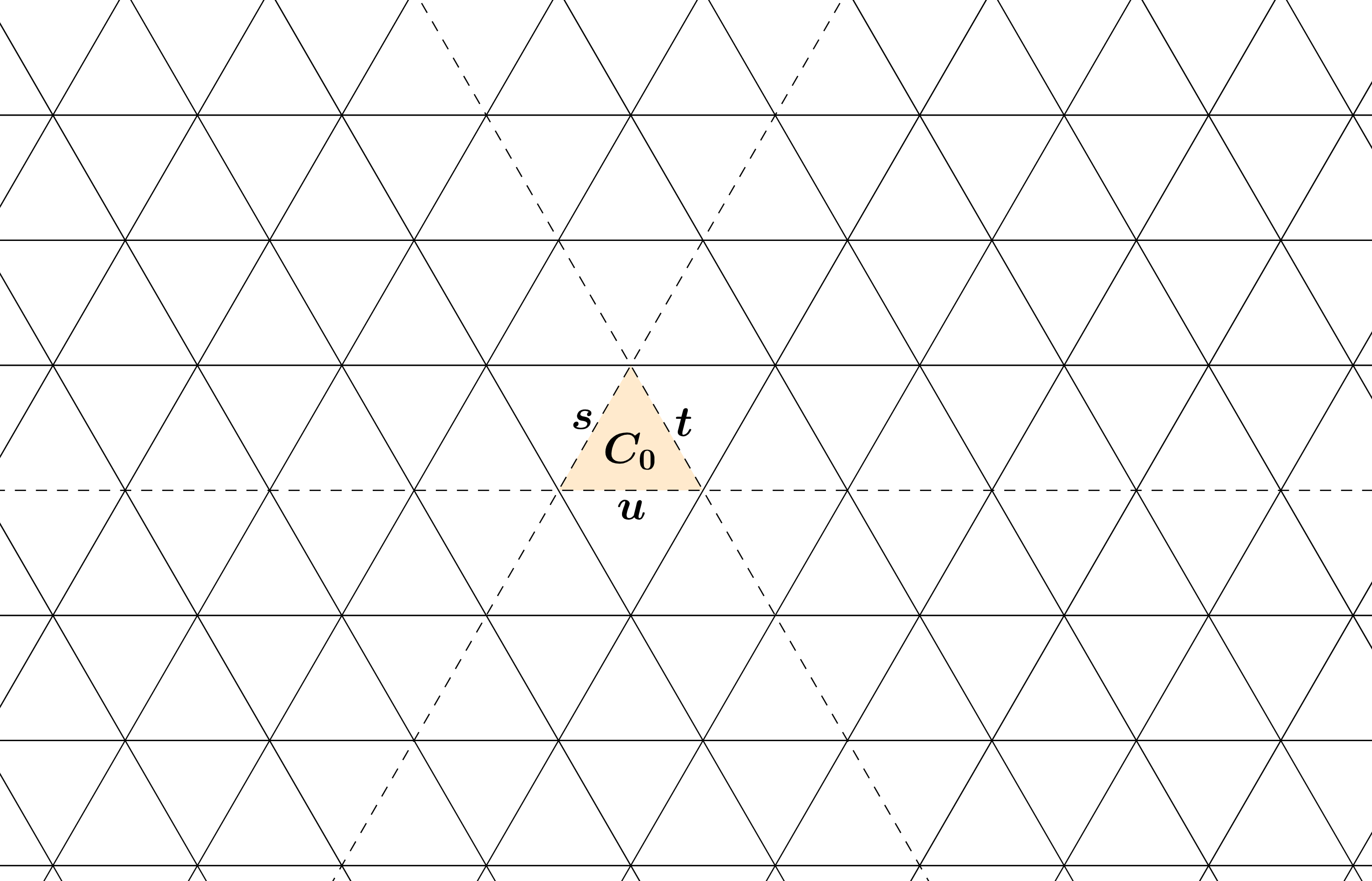}
  \captionof{figure}{$W=\langle s,t,u\rangle$ of affine type $(3,3,3)$.}
  \label{figure:A2tilde}
\end{minipage}%
\begin{minipage}[t]{.5\textwidth}
  \centering
  \includegraphics[trim = 43mm 14mm 17mm 14mm, clip, width=6.5cm]{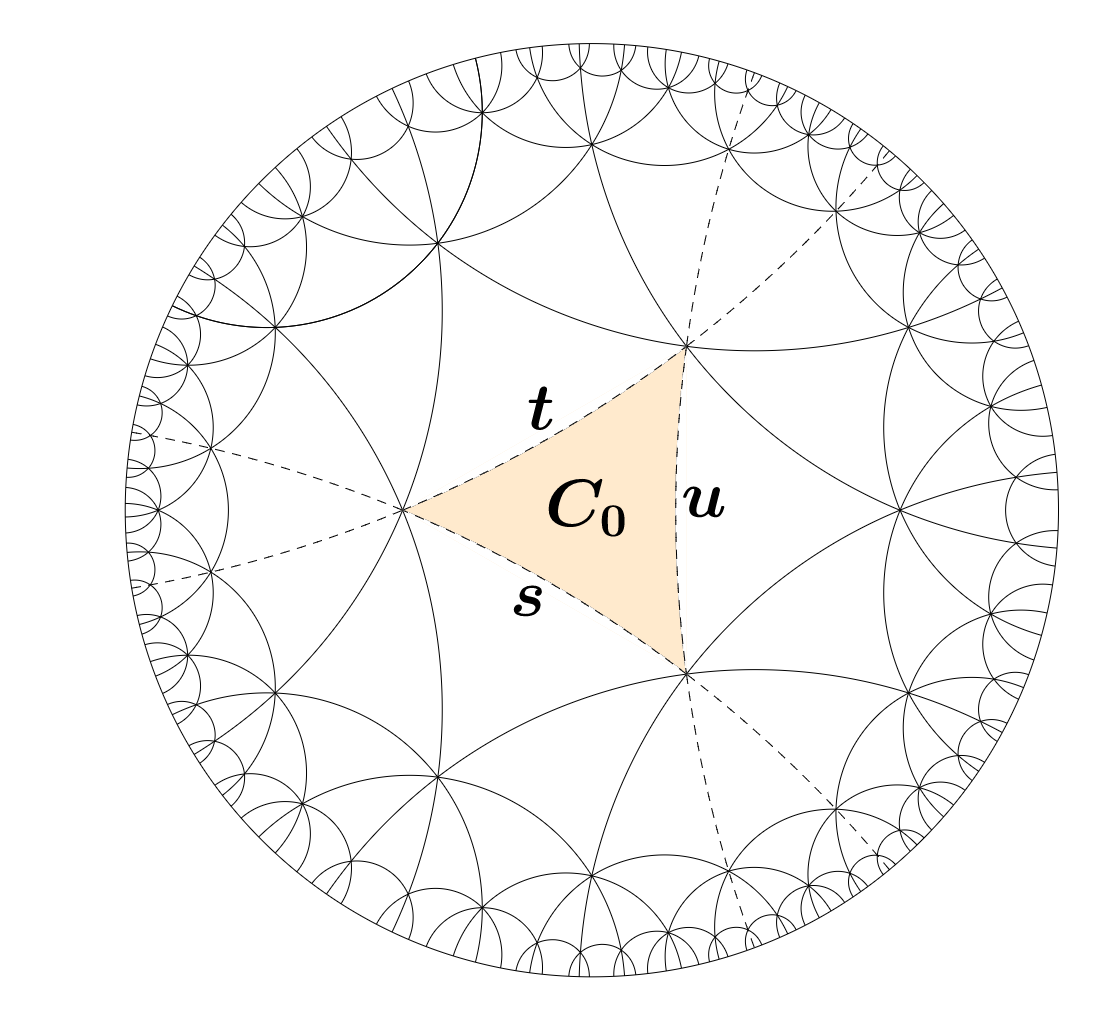}
  \captionof{figure}{$W=\langle s,t,u\rangle$ of compact hyperbolic type $(4,4,4)$.}
  \label{figure:hyperbolic}
\end{minipage}
\end{figure}

Each chamber $C$ of $\Sigma$ is realised in $X$ as a (closed) convex subset: if $C=wC_0$, so that $x_C:=\{w\}$ is a vertex in $X^{(1)}$, then each spherical residue $wR_I$ ($I\subseteq S$ spherical) containing $C$ yields a polytope $P_I$ in $X$ for which $x_I=x_C$, and $C$ is then realised as the union over all such polytopes of the intersections $P_I\cap C_I$. We call the vertex $x_C$ the {\bf barycentre} of $C$. We identify the chambers, walls and half-spaces of $\Sigma$ with their realisation in $X$. In particular, we will also write $\Ch(X)$ instead of $\Ch(\Sigma)$.

Alternatively, each wall $m$ of $\Sigma$ is realised as the fixed-point set of $r_m$ in $X$, and delimits two connected components of $X$, whose closures are the realisations of the corresponding half-spaces. The chamber $C$ of $\Sigma$ is the intersection of half-spaces associated to each of its walls, and its realisation in $X$ is then the same intersection of (closed) half-spaces, viewed as subsets of $X$.

The walls and half-spaces have strong convexity properties:

\begin{lemma}\label{lemma:convexity_walls}
Let $r\subseteq X$ be a geodesic segment and $m$ a wall. 
\begin{enumerate}
\item
If $r$ intersects $m$ in at least two points, then $r\subseteq m$.
\item
The open half-spaces delimited by $m$ are convex. In particular, if $r$ intersects $m$ and is contained in one of the two closed half-spaces delimited by $m$, then $r\subseteq m$.
\end{enumerate}
\end{lemma}
\begin{proof}
(1) is \cite[Lemma~2.2.6]{Nos11} and (2) is \cite[Lemma~2.3.1]{Nos11} (the second statement of (2) follows from the first and (1)).
\end{proof}

Given two points $x,y\in X$, we denote by $[x,y]$ the unique geodesic in $X$ from $x$ to $y$, and we set $]x,y[:=[x,y]\setminus\{x,y\}$, and similarly for $[x,y[$ and $]x,y]$. In Section~\ref{section:CGWNIPPS}, we will restrict our study of $X$ to the case where all proper parabolic subgroups of $W$ are finite. In such cases, $X$ has the additional property that all geodesic segments can be extended, and this is the extra assumption we will need in order to establish the results of that section:

\begin{lemma}\label{lemma:extgeod}
Assume that $W$ is infinite irreducible but all proper parabolic subgroups of $W$ are finite. 
\begin{enumerate}
\item
Let $x$ be the barycentre of a chamber $C$, let $m_1,\dots,m_k$ be the walls of $C$, and for each $i=1,\dots,k$, choose a point $x_i\in [x,r_{m_i}(x)]\setminus \{x\}$. Then the convex hull of $\{x_1,\dots,x_k\}$ contains an open neighbourhood of $x$.
\item 
Every geodesic segment in $X$ can be extended to a geodesic line.
\end{enumerate}
\end{lemma}
\begin{proof}
(1) Note that for each $j=1,\dots,k$, the set $\{x,r_{m_i}(x) \ | \ i\neq j\}$ is a subset of the set of vertices of a maximal (closed) Coxeter cell $\sigma_j$ of $X$ of dimension $k-1$. In particular, the convex hull of $\{x, x_i \ | \ i\neq j\}$ contains an open neighbourhood $V_j$ of $x$ in $\sigma_j$. As $x$ lies in the convex hull of $\{r_{m_1}(x),\dots,r_{m_k}(x)\}$, it lies in the interior of $\bigcup_{j=1}^k\sigma_j$. The union $\bigcup_{j=1}^kV_j$ then yields the desired neighbourhood of $x$ in $X$.

(2) Note that for each non-maximal spherical subset $J\subseteq S$ (so that $|J|\leq |S|-2$, say $s,t\in S\setminus J$ with $s\neq t$), $J$ is properly contained in at least two distinct spherical subsets of $S$ (namely, $J\cup\{s\}$ and $J\cup\{t\}$). This implies that $X$ has \emph{no free faces} in the sense of \cite[II.5.9]{BHCAT0}. Hence the claim follows from \cite[II.5.8 and 5.10]{BHCAT0}.
\end{proof}

Note that Lemma~\ref{lemma:extgeod}(2) in general fails for arbitrary infinite irreducible $W$.


\section{Coxeter groups with an infinite proper parabolic subgroup}\label{section:CGWAIPPS}
Let $(W,S)$ be a Coxeter system and $\Sigma$ its Coxeter complex.

\begin{lemma}\label{lemma:sphericalcriterion}
Let $H\subseteq S$ be spherical, and let $x\in W^H$. Set $J:=\{s\in S \ | \ \ell(sx)=\ell(x)+1\}$. Suppose that $sx\notin W^H$ for all $s\in J$. Then $S$ is spherical. 
\end{lemma}
\begin{proof}
Let $R:=R_J$ be the standard $J$-residue, and $R':=xR_H$ the $H$-residue containing $xC_0$. Since $x\in W^H$ by assumption, $xC_0=\proj_{R'}(C_0)$. 

Let $s\in J$. Since $sx\notin W^H$, 
$$sxC_0\neq\proj_{sxR_H}(C_0)=s\proj_{R'}(sC_0),$$
and hence $D_s:=\proj_{R'}(sC_0)\neq xC_0$ (see Figure~\ref{figure:sphercrit}). This implies that the wall $m_s$ fixed by $s$ (that is, the wall separating $C_0$ from $sC_0$) coincides with the wall separating $xC_0=\proj_{R'}(C_0)$ from $D_s=\proj_{R'}(sC_0)$ (recall that $\proj_{R'}$ does not increase $\dc$, so that $xC_0$ and $D_s$ are adjacent): indeed, suppose for a contradiction that $xC_0$ and $D_s$ lie on the same side of $m_s$, say on the same side of $m_s$ as $C_0$ (otherwise we exchange the roles of $C_0$ and $sC_0$). Then the projection of $D_s$ on the standard $s$-residue $\{C_0,sC_0\}$ is $C_0$, that is, $\dc(D_s,sC_0)=\dc(D_s,C_0)+1$. On the other hand, the gate property implies that
$$\dc(D_s,sC_0)=\dc(xC_0,sC_0)-1\leq\dc(xC_0,C_0)=\dc(D_s,C_0)-1,$$
a contradiction. We have thus showed that $D_s=sxC_0$. Let $\alpha(s)\in H$ be such that $D_s$ and $xC_0$ are $\alpha(s)$-adjacent, so that $D_s=x\alpha(s)C_0$. This defines a map $\alpha\co J\to H$ such that $\alpha(s)=x\inv sx$ for all $s\in J$.

Let now $\overline{R}:=xR_{\alpha(J)}\subseteq R'$ be the $\alpha(J)$-residue containing $xC_0$. Then the residues $R$ and $\overline{R}$ have the same stabiliser $W_J=xW_{\alpha(J)}x\inv$ (as $x\alpha(J)x\inv=J$), and are thus parallel. Moreover, $x=\delta(R,\overline{R})$ as $xC_0=\proj_{R'}(C_0)=\proj_{\overline{R}}(C_0)$. Finally, as $\alpha(J)\subseteq H$ and $H$ is spherical by assumption, both residues $R$ and $\overline{R}$ are spherical. Lemma~\ref{lemma:MPWcriterionspherical} then implies that $S$ is spherical, as desired. 
\end{proof}

\begin{figure}

  \centering
  \includegraphics[trim = 35mm 28mm 43mm 40mm, clip, width=12cm]{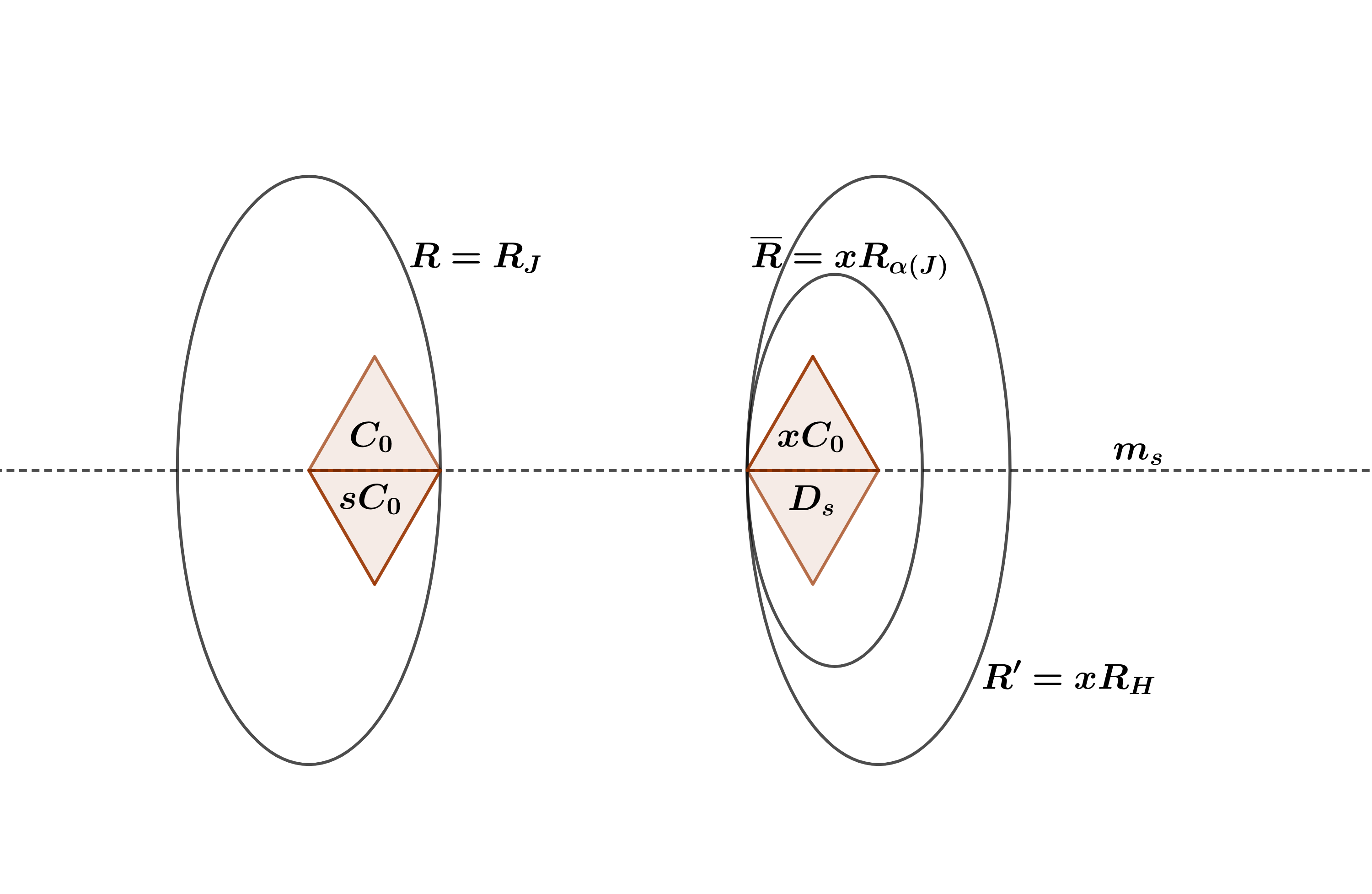}
  \captionof{figure}{Lemma~\ref{lemma:sphericalcriterion}}
  \label{figure:sphercrit}

\end{figure}

\begin{definition}
Given $w,w'\in W$ and $s\in S$, write $w\stackrel{s}{\leftarrow}w'$ if $w'=sws$ and $\ell(w)\leq\ell(w')$. For $I\subseteq S$, write $w\stackrel{I}{\leftarrow} w'$ if there is a sequence  $w=w_0\stackrel{s_1}{\leftarrow}w_1\stackrel{s_2}{\leftarrow}\dots \stackrel{s_k}{\leftarrow}w_k=w'$ for some $s_1,\dots,s_k\in I$. 
We further set
\begin{gather*}
  U^+_I(w):=\{w'\in W \ | \ w\stackrel{I}{\leftarrow} w'\}
  \text{.}
\end{gather*}
\end{definition}

\begin{lemma}\label{lemma:keylemma}
Let $w\in W$. Suppose there exists a non-spherical subset $I\subseteq S$ such that $H:=I\cap \overline{w}I\overline{w}\inv$ is spherical, where $\overline{w}\in \prescript{I}{}{W}^I$ is the unique element of minimal length in $W_IwW_I$. Then $U^+_I(w)$ is infinite.
\end{lemma}
\begin{proof}
Suppose for a contradiction that $U^+_I(w)$ is finite. Up to modifying $w$ inside $U^+_I(w)$, we may then assume $w$ to be of maximal length in $U^+_I(w)$ --- note that $U^+_I(w)\subseteq W_IwW_I=W_I\overline{w}W_I$. By Lemma~\ref{lemma:doubleparabdecomp}, we have a unique decomposition $w=xh\overline{w}y$ with $h\in W_H$, $x\in W_I^H$ and $y\in \prescript{\overline{H}}{}{W}_I$, where $\overline{H}:=\overline{w}\inv I\overline{w}\cap I=\overline{w}\inv H\overline{w}$, and we have $\ell(w)=\ell(x)+\ell(h)+\ell(\overline{w})+\ell(y)$. Up to further modifying $w$ inside $U^+_I(w)$, we may assume that $\ell(y)$ is minimal, and then that $\ell(x)$ is maximal (among elements of $U^+_I(w)$ with minimal $\ell(y)$).

Since $I$ is not spherical by assumption, Lemma~\ref{lemma:sphericalcriterion} yields some $s\in I$ with $\ell(sx)=\ell(x)+1$ such that $sx\in W_I^H$. Thus $\ell(sxh)=\ell(sx)+\ell(h)=\ell(x)+\ell(h)+1$. Moreover, $\ell(sw)=\ell(sxh)+\ell(\overline{w})+\ell(y)=\ell(w)+1$ by Lemma~\ref{lemma:doubleparabdecomp}, so that $\ell(sws)\geq\ell(w)$ and $sws\in U^+_I(w)$.

Note that $\ell(ys)=\ell(y)+1$, for otherwise $sws=sxh\overline{w}ys$ with $ys\in \prescript{\overline{H}}{}{W}_I$ of length $\ell(ys)<\ell(y)$, contradicting the minimality assumption on $\ell(y)$.

This implies that $ys\notin \prescript{\overline{H}}{}{W}_I$, for otherwise $\ell(sws)=\ell(sxh\overline{w}ys)=\ell(sxh)+\ell(\overline{w})+\ell(ys)=\ell(w)+2$ again by Lemma~\ref{lemma:doubleparabdecomp}, contradicting the maximality assumption on $\ell(w)$.

Write $ys=s_{\overline{H}}y'$ for some $s_{\overline{H}}\in \overline{H}$ and $y'\in W_I$ with $\ell(y')=\ell(y)$. Then the \emph{deletion condition} (cf. \cite[\S2.3]{BrownAbr}) implies that one can get a reduced expression for $y$ by deleting two letters in its non-reduced expression $s_{\overline{H}}\boldsymbol{y'}s$, where $\boldsymbol{y'}$ is a reduced expression for $y'$. One of these letters must be $s$ (because $s_{\overline{H}}\boldsymbol{y'}$ is reduced) and the other must be $s_{\overline{H}}$ (because $y\in \prescript{\overline{H}}{}{W}_I$). Thus, $y'=y$.  Set $s_H:=\overline{w}s_{\overline{H}}\overline{w}\inv\in H$. Then
$$sws=sxh\overline{w}ys=sxh\overline{w}s_{\overline{H}}y=(sx)hs_H\overline{w} y$$
with $hs_H\in H$, $sx\in W_I^H$ and $y\in \prescript{\overline{H}}{}{W}_I$, contradicting the maximality assumption on $\ell(x)$.
\end{proof}

\begin{lemma}\label{lemma:wdoesnotnormalisethenUinf}
Let $w\in W$. Suppose there exists a non-spherical subset $I\subseteq S$ all whose proper subsets are spherical, such that $w$ does not normalise $W_I$. Then $U^+_I(w)$ is infinite.
\end{lemma}
\begin{proof}
Note that if $\overline{w}$ is the minimal length element of $W_IwW_I$, then $\overline{w}I\overline{w}\inv\neq I$ for otherwise $w$ would normalise $W_I$. Hence $H:=I\cap \overline{w}I\overline{w}\inv$ is spherical. The conclusion then follows from Lemma~\ref{lemma:keylemma}.
\end{proof}

\begin{prop}\label{prop:U+Swinfinitenoncompacthyp}
Let $(W,S)$ be irreducible. Assume that not all proper subsets of $S$ are spherical. Then $U^+_S(w)$ is infinite for all $w\in W\setminus\{1\}$.
\end{prop}
\begin{proof}
By assumption, there is an irreducible non-spherical subset $I\subsetneq S$. Without loss of generality, we may further assume all proper subsets of $I$ are spherical. 

Let $w\in W\setminus\{1\}$. Note that if $s\notin\supp(w)$ (where $\supp(w)$ denotes the support of $w$, namely the letters from $S$ appearing in any reduced decomposition of $w$), and if $s$ does not commute with every element of $\supp(w)$, then $\ell(sws)=\ell(w)+2$ (see e.g. \cite[Lemma~2.37]{BrownAbr}). Up to modifying $w$ inside $U^+_S(w)$, we may thus assume that $\supp(w)=S$. 

On the other hand, $N_W(W_I)=W_I\times W_{I^{\perp}}=W_{I\cup I^{\perp}}$, where $I^{\perp}:=\{s\in S \ | \ st=ts \ \forall t\in I\}$ (see e.g. \cite[Lemma~2.1]{openKM}). As $I\cup I^{\perp}\neq S$ (because $S$ is irreducible), we conclude that $w\notin N_W(W_I)$. 
The conclusion then follows from Lemma~\ref{lemma:wdoesnotnormalisethenUinf}.
\end{proof}


\section{The affine and compact hyperbolic cases}\label{section:CGWNIPPS}
Let $(W,S)$ be an infinite irreducible Coxeter system, and let $X$ be its Davis complex. 

\begin{definition}
Let $w\in W$. To each chamber $C=vC_0\in\Ch(X)$ ($v\in W$), we associate the conjugate $\pi_w(C):=v\inv wv$ of $w$. 
We call a gallery $\Gamma=(D_0,D_1,\dots,D_k)$ in $X$ {\bf $w$-non-decreasing} if $\ell(\pi_w(D_i))\leq\ell(\pi_w(D_{i+1}))$ for all $i$. We call $\Gamma$ {\bf $w$-flat} if $\ell(\pi_w(D_i))=\ell(\pi_w(D_{i+1}))$ for all $i$.
Finally, we call $C\in\Ch(X)$ a {\bf $w$-bad} chamber if every $w$-non-decreasing gallery from $C$ is $w$-flat.
\end{definition}

\begin{lemma}\label{lemma:geometrictranslations}
Let $w\in W$ and $\Gamma=(D_0,D_1,\dots,D_k)$ be a gallery in $X$ of type $(s_1,\dots,s_k)$. Then:
\begin{enumerate}
\item $\Gamma$ is $w$-non-decreasing if and only if $\pi_w(D_{i-1})\stackrel{s_i}{\leftarrow}\pi_w(D_{i})$ for each $i=1,\dots,k$.
\item If no chamber of $X$ is a $w$-bad chamber, then $U^+_S(w')$ is infinite for all conjugates $w'\in W$ of $w$.
\end{enumerate}
\end{lemma}
\begin{proof}
For (1), note that two chambers $C,D$ are $s$-adjacent ($s\in S$) if and only if $C=vC_0$ and $D=vsC_0$ for some $v\in W$ if and only if $\pi_w(D)=s\pi_w(C)s$.

For (2), suppose $U^+_S(w')$ is finite for some conjugate $w'\in W$ of $w$. Without loss of generality, we may assume $w'$ to be of maximal length inside $U^+_S(w')$. Then by definition and in view of (1), if $C\in\Ch(X)$ is such that $\pi_w(C)=w'$, every $w$-non-decreasing gallery from $C$ is flat, i.e. $C$ is a bad chamber.
\end{proof}

\begin{definition}
For $w\in W$, define the function
$$\dist_w\co X\to\RR:x\mapsto\dist(x,wx).$$ 
Note that $\dist_w$ is a convex function (see \cite[II,6.2(3)]{BHCAT0}), that is, for each geodesic $\gamma\co I\to X$ defined on the interval $I\subseteq\RR$, the function $\dist_w\circ\gamma\co I\to \RR$ is convex.
\end{definition}

Given a subset $C$ of $X$, we write $\Int(C)$ for its interior.

\begin{prop}\label{prop:dvsdchvsdw}
Let $w\in W$. Let $C$ be a chamber, $m$ a wall of $C$, and set $D:=r_m(C)$. Assume that $\dc(C,wC)<\dc(D,wD)$. Let $x\in \Int(C)$ and set $y:=r_m(x)\in \Int(D)$. Set also $z:=[x,y]\cap m$. Then $\dist_w$ is strictly increasing on $[z,y]$.
\end{prop}
\begin{proof}
Let $\sigma:=\Int(C\cap D)$. As $\Int(C)\cup\Int(D)\cup\sigma$ is convex (i.e. it is an intersection of open half-spaces), $[x,y]$ intersects $\sigma$, at $z$. Note that
\begin{equation}
wm\neq m,
\end{equation}
for otherwise $r_m=wr_mw\inv$ and hence $\dc(D,wD)=\dc(C,r_mwr_m(C))=\dc(C,wC)$, a contradiction. In particular, $]z,wz[$ does not intersect $\sigma$, for otherwise $[z,wz]$ would be entirely contained in $m$ by Lemma~\ref{lemma:convexity_walls}(1), contradicting the fact that $wm$ is the only wall containing $wz$. Thus $]z,wz[$ intersects either $\Int(C)$ or $\Int(D)$. But the latter case is not possible, for in that case $D$ would lie in the same half-space delimited by $m$ as $wz$ and hence (since $wz\notin m$) as $wD$, contradicting the fact that $m$ separates $D$ from $wD$ (because $\dc(C,wC)<\dc(D,wD)$ by assumption). Thus $]z,wz[$ intersects $\Int(C)$.

In particular, if we move a bit $z$ along $[x,y]$ towards $y$, say at $x'\in\Int(D)$, the geodesic $[x',wx']$ still passes through $\Int(C)$, and hence intersects $\sigma$, say at $a$. For the same reasons, up to choosing $x'$ closer to $z$ (and arguing with $(w\inv,wC,wm)$ instead of $(w,C,m)$), the geodesic $[x',wx']$ also intersects $w\sigma$, say at $b$ (see Figure~\ref{figure:dvsdchvsdw}). Set $x'':=r_m(x')$. Then
\begin{align*}
\dist_w(x'')&=\dist(x'',wx'')\leq\dist(x'',a)+\dist(a,b)+\dist(b,wx'')\\
&=\dist(x',a)+\dist(a,b)+\dist(b,wx')=\dist(x',wx')=\dist_w(x').
\end{align*}
In fact, the above inequality is even strict, since the concatenation of $[x'',a]$ and $[a,b]$ cannot be a geodesic by Lemma~\ref{lemma:convexity_walls}(2). Thus $\dist_w(x'')<\dist_w(x')$. Since the above argument can be repeated if we replace $x'$ by any point on $[x',z[$, the convexity of $\dist_w$ then implies that $\dist_w$ is strictly increasing on $[z,y]$, as desired.
\end{proof}

\begin{figure}

  \centering
  \includegraphics[trim = 7mm 3mm 1mm 14mm, clip, width=13cm]{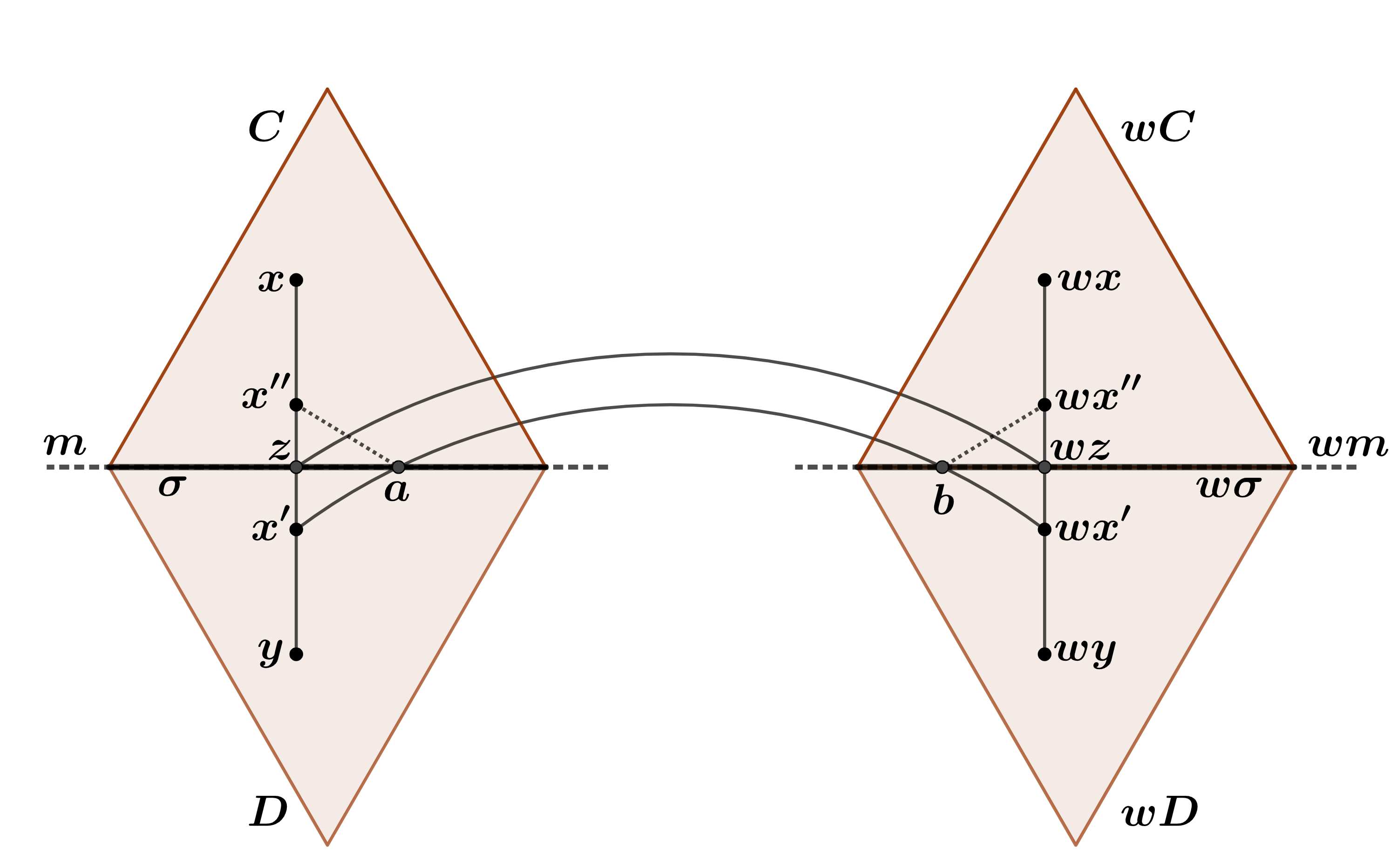}
  \captionof{figure}{Proposition~\ref{prop:dvsdchvsdw}}
  \label{figure:dvsdchvsdw}

\end{figure}

\begin{theorem}\label{thm:nobadallbad}
Assume that every proper parabolic subgroup of $W$ is finite. Let $w\in W$. Suppose that there exists a $w$-bad chamber $C\in\Ch(X)$. Then all the conjugates of $w$ have the same length.
\end{theorem}
\begin{proof}
Let $C$ be a $w$-bad chamber and let $C_{=}$ denote the set of chambers on a $w$-flat gallery from $C$ (in particular, all chambers of $C_=$ are $w$-bad). For each chamber $D$, write $x_D\in\Int(D)$ for its barycentre. 

Note that $\{\dist_w(x_D) \ | \ D\in C_{=}\}$ is finite. Indeed, since $\ell(\pi_w(D))=\ell(\pi_w(C))$ for all $D\in C_=$, the set $\{\pi_w(D) \ | \ D\in C_{=}\}$ is finite. On the other hand, for two chambers $D_1,D_2$, we have $\pi_w(D_1)=\pi_w(D_2)$ if and only if $D_1,D_2$ are in the same $\ZZZ_W(w)$-orbit (where $\ZZZ_W(w)$ denotes the centraliser of $w$ in $W$). Hence $C_=$ is covered by finitely many $\ZZZ_W(w)$-orbits of chambers. As $\dist_w$ is constant on $\ZZZ_W(w)$-orbits, the claim follows.

Up to modifying $C$ inside $C_=$, we may now assume that $\dist_w(x_C)$ is maximal among the chambers of $C_=$. Set for short $x:=x_C$. Let $m_1,\dots,m_k$ be the walls of $C$, and set $D_i:=r_{m_i}(C)$ for each $i$. By assumption, for each $i$, either $D_i\in C_{=}$, or $\dc(C,wC)>\dc(D_i,wD_i)$. In the second case, Proposition~\ref{prop:dvsdchvsdw} implies that $\dist_w$ is strictly increasing on $[x_i,x]$, where $x_i:=m_i\cap [x_{D_i},x]$. In the first case, we have $\dist_w(x_{D_i})\leq \dist_w(x)$ by the maximality assumption on $C$. Since $\dist_w$ is convex, we then either have that $\dist_w$ is strictly increasing on $[x_i,x]$ for some $x_i\in [x_{D_i},x[$, or else $\dist_w$ is constant on $[x_{D_i},x]$ (in which case we set $x_i:=x_{D_i}$). 

We have thus defined for each $i=1,\dots,k$ a point $x_i\in [x_{D_i},x[$ such that either $\dist_w$ is strictly increasing on $[x_i,x]$, or $x_i=x_{D_i}$ and $\dist_w$ is constant on $[x_i,x]$ and $D_i\in C_=$. We claim that only the second possibility can occur. Indeed, suppose for a contradiction that $\dist_w$ is strictly increasing on $[x_j,x]$ for some $j$. By  Lemma~\ref{lemma:extgeod}, we can extend $[x_j,x]$ to a geodesic segment $[x_j,y]$ (with $[x_j,x]\subsetneq [x_j,y]$, so that $\dist_w(x)<\dist_w(y)$ by convexity of $\dist_w$) such that $y$ lies in the convex hull of $\{x_1,\dots,x_k\}$. Since $\dist_w(x_i)\leq\dist_w(x)$ for all $i$, the convexity of $\dist_w$ implies that $\dist_w(y)\leq\dist_w(x)$, a contradiction.

Thus, for any chamber $D$ adjacent to $C$, we have $D\in C_=$ and $\dist_w(x_D)=\dist_w(x_C)$ (in particular, $D$ also satisfies the maximality assumption satisfied by $C$). We may thus repeat the above argument with $C$ replaced by any of its adjacent chambers, and hence also inductively to conclude that $C_==\Ch(X)$. In particular, $\dc(D,wD)=\ell(\pi_w(D))$ is independent of $D\in\Ch(X)$, as desired.
\end{proof}

The following lemma is well-known; as we could not find it explicitely stated in the literature, we include here a proof for the benefit of the reader.
\begin{lemma}\label{lemma:ICC}
Suppose $W$ is infinite irreducible, and let $w\in W\setminus\{1\}$. Then the conjugacy class of $w$ is finite if and only if $W$ is of affine type and $w$ is a translation.
\end{lemma}
\begin{proof}
The implication $\Leftarrow$ is clear. For the forward implication, suppose that the conjugacy class of $w$ is finite. Then in fact all conjugates of $w$ have the same length (see \cite[Theorem~1.3]{Badupward}). In particular, $w$ has minimal length in its conjugacy class, and hence its parabolic closure $\Pc(w)$ (the smallest parabolic subgroup containing $w$) is standard (see e.g. \cite[Proposition~4.2]{CF10}). This implies that $\Pc(w)=W$, for if $w\in W_I$ for some $I\subsetneq S$ then, as in the proof of Proposition~\ref{prop:U+Swinfinitenoncompacthyp}, there is some $s\in S\setminus I$ such that $\ell(w)<\ell(sws)$.

Assume first that $W$ is not of affine type. By assumption, $\ZZZ_W(w)$ has finite index in $W$. Hence \cite[Corollary~6.3.10]{Kra09} implies that the subgroup $\langle w\rangle$ generated by $w$ has finite index in $W$, so that $W$ is virtually abelian, a contradiction. 

Assume now that $W$ is of affine type, so that $X$ is a Euclidean space. In that case, $W$ decomposes as a semidirect product $W=W_{x}\ltimes T$, where $T$ is a group of translations of $X$ and $W_x$ is a finite parabolic subgroup (the fixer of a \emph{special} vertex $x\in X$) --- see e.g. \cite[\S 10.1.6]{BrownAbr}. Write $w=ut$ accordingly, with $u\in W_x$ and $t\in T$. Suppose for a contradiction that $w$ is not a translation, i.e. $u\neq 1$. The first paragraph implies that the conjugacy class of $u$ in $W$ is infinite. Let $(v_n)_{n\in\NN}\subseteq W$ be such that $\ell(v_n\inv uv_n)\stackrel{n\to\infty}{\to}\infty$. Write $v_n=u_nt_n$ with $u_n\in W_x$ and $t_n\in T$. Up to extracting a subsequence, we may assume that $(u_n)_{n\in\NN}$ is constant, say $u_n=u_0$ for all $n$. Set $u':=u_0\inv uu_0\in W_x$,  $t':=u_0\inv tu_0\in T$ and $w':=u't'=u_0\inv wu_0$. As $\ell(t_n\inv u't_n)\stackrel{n\to\infty}{\to}\infty$ and $t_n\inv w't_n=t_n\inv u't_n\cdot t'$, we also have $\ell(t_n\inv w't_n)\stackrel{n\to\infty}{\to}\infty$, so that $w'$ (and hence $w$) has an infinite conjugacy class, a contradiction.
\end{proof}

\begin{corollary}\label{cor:U+Swinfinitecompacthyp}
Assume that every proper parabolic subgroup of $W$ is finite. Let $w\in W\setminus\{1\}$, and assume $w$ is not a translation if $W$ is of affine type. Then $U^+_S(w)$ is infinite. 
\end{corollary}
\begin{proof}
Lemma~\ref{lemma:ICC} implies that the conjugacy class of $w$ is infinite. Theorem~\ref{thm:nobadallbad} then implies that no chamber of $X$ is $w$-bad. The conclusion thus follows from Lemma~\ref{lemma:geometrictranslations}(2).
\end{proof}


\section{The centre of Hecke algebras}
\label{sec:center}

We follow Garrett's presentation of generic Hecke algebras from \cite[Chapter 6.1]{garrett97}.  Given a Coxeter system $(W,S)$ and a commutative ring $R$, let $(a_s, b_s)_{s \in S}$ be a family of elements from $R$ subject to the relation that $(a_s, b_s) = (a_t,b_t)$ if $s$ and $t$ are conjugate in $W$.  Such tuples are termed \textbf{deformation parameters}.  Denote by $\HHH = \HHH(W, S, (a_s, b_s)_{s \in S})$ the \textbf{generic Hecke algebra} which is the free $R$-module with basis $(T_w)_{w \in W}$ equipped with the unique $R$-algebra structure satisfying
\begin{align*}
  T_s T_w & = T_{sw}  & \text{ if } \ell(sw) > \ell(w) \\
  T_s T_w & = a_s T_w + b_s T_{sw} & \text{ if } \ell(sw) < \ell(w) \\
  T_w T_s & = T_{ws}  & \text{ if } \ell(ws) > \ell(w) \\
  T_w T_s & = a_s T_w + b_s T_{ws} & \text{ if } \ell(ws) < \ell(w) &
                                   \text{.}
\end{align*}

\begin{example}
  \label{ex:iwahori-hecke-algebra}
  The most common specialisation of generic Hecke algebras are Iwahori-Hecke algebras, either over the complex numbers or the ring $\ZZ[v,v^{-1}]$. The former is the basis of operator algebraic considerations in \cite{caspersklisselarsen19,raumskalski20,klisse2021} while the latter is prominent in the study of finite groups of Lie type and $p$-adic reductive groups \cite{KL79,Lus03}.

  More precisely, if $\mathbf{q} = (q_s)_{s \in S} \in \RR_{> 0}^{S}$ is a multiparameter satisfying $q_s = q_t$ whenever $s$ and $t$ are conjugate in $W$, then $\CC_{\mathbf{q}}[W]$ is the generic Hecke algebra over $\CC$ with parameters $(q_s^{1/2} - q_s^{-1/2}, 1)_{s \in S}$, that is
  \begin{align*}
    T_s T_w & = T_{sw}  & \text{ if } \ell(sw) > \ell(w) \\
    T_s T_w & = (q_s^{1/2} - q_s^{-1/2}) T_w + T_{sw} & \text{ if } \ell(sw) < \ell(w) &
    \text{.}
  \end{align*}
  The algebras $\CC_{\mathbf{q}}[W]$ naturally appear in the study of representation theory of groups acting on buildings and are the basis for the construction of Hecke operator algebras once equipped with a suitable *-structure.

  The Iwahori-Hecke algebras considered over $\ZZ[v,v^{-1}]$ are associated with the Coxeter system $(W,S)$ equipped with an additional {\bf length function} $L: W \to \ZZ$ satisfying $L(w_1w_2) = L(w_1) + L(w_2)$ whenever $\ell(w_1w_2) = \ell(w_1) + \ell(w_2)$.  The Hecke algebra over $\ZZ[v,v^{-1}]$ with parameters $(v^{L(s)},1)_{s\in S}$ is studied in representation theory of finite and reductive groups.
\end{example}

Let $\HHH$ be a generic Hecke algebra.  For $x \in \HHH$ we write $x = \sum_{w \in W} x_w T_w$ for its expansion in the standard basis of $\HHH$. In order to describe the multiplication by Hecke operators on the coefficients $(x_w)_w$, we consider the functionals $\delta_w \in \Hom_R(\HHH, R)$ satisfying $\delta_w(T_v) = \delta_{v,w}$. 

For $s \in S$ we denote by $\lambda_s, \rho_s \in \End_R(\HHH)$ the endomorphisms given by left and right multiplication with $T_s$, respectively.  We denote by $\lambda_s', \rho_s' \in \End_R(\Hom(\HHH, R))$ the dual endomorphisms satisfying
\begin{gather*}
  (\lambda_s' \varphi)(x)  = \varphi(T_sx)
  \quad \text{and} \quad 
  (\rho_s' \varphi)(x)  = \varphi(xT_s)
             \quad \text{for all $\varphi\in\Hom(\HHH, R)$ and $x\in \HHH$.}
\end{gather*}

\begin{lemma}
  \label{lem:dual-action}
  The following formulas describe the action of $\lambda_s'$ and $\rho_s'$ on the dual basis $(\delta_w)_{w \in W}$.
  \begin{align*}
    \lambda_s' \delta_w
    & =
    \begin{cases}
      b_s \delta_{sw} & \text{if } \ell(sw) > \ell(w) \text{,} \\
      a_s \delta_w + \delta_{sw} & \text{if } \ell(sw) < \ell(w) \text{.}
    \end{cases} \\
    \rho_s' \delta_w
    & =
    \begin{cases}
      b_s \delta_{ws} & \text{if } \ell(ws) > \ell(w) \text{,} \\
      a_s \delta_w + \delta_{ws} & \text{if } \ell(ws) < \ell(w) \text{.}
    \end{cases}
  \end{align*}
\end{lemma}
\begin{proof}
    Let $v,w \in W$ and $s \in S$.  Then
  \begin{align*}
    \lambda_s' \delta_w(T_v)
    & = \delta_w(T_s T_v) \\
    & =
      \begin{cases}
        \delta_{w}(T_{sv}) & \text{if } \ell(sv) > \ell(v) \\
        \delta_{w}(a_s T_v + b_s T_{sv}) & \text{if } \ell(sv) < \ell(v)
      \end{cases} \\
    & =
      \begin{cases}
        1 & \text{if } \ell(sv) > \ell(v) \text{ and } sv = w \\
        a_s  & \text{if } \ell(sv) < \ell(v) \text{ and } v = w \\
        b_s  & \text{if } \ell(sv) < \ell(v) \text{ and } sv = w \text{.}
      \end{cases}
  \end{align*}
  Since $v \in W$ was arbitrary, this can be summarised as
  \begin{gather*}
    \lambda_s' \delta_w =
    \begin{cases}
      b_s \delta_{sw} & \text{if } \ell(sw) > \ell(w) \\
      a_s \delta_w + \delta_{sw} & \text{if } \ell(sw) < \ell(w) \text{.}
    \end{cases}
  \end{gather*}
  The formula for $\rho_s' \delta_w$ is obtained in a similar way.
\end{proof}

We are now ready to prove Theorem~\ref{thmintro:Hecke-generic} (and hence also Corollary~\ref{corintro:Hecke}).
\begin{theorem}
  \label{thm:center-generic-hecke-algebras}
  Assume that $(W,S)$ is of irreducible indefinite type. 
   Then the centre of $\HHH=\HHH(W, S, (a_s, b_s)_{s \in S})$ is trivial.
\end{theorem}
\begin{proof}  
  Assume for a contradiction that $x \in\HHH$ is a central element that is nontrivial.  Since the set $\{w \in W \ | \ x_w \neq 0\}$ is finite, it contains an element of maximal length, say $w\neq 1$.  By Proposition~\ref{prop:U+Swinfinitenoncompacthyp} and Corollary~\ref{cor:U+Swinfinitecompacthyp}, the set $U^+_S(w)$ is infinite. In particular, there is a sequence $w=w_0\stackrel{s_1}{\leftarrow}w_1\stackrel{s_2}{\leftarrow}\dots \stackrel{s_k}{\leftarrow}w_k=w'$ with $\ell(w_0)=\ell(w_1)=\dots=\ell(w_{k-1})<\ell(w')$, which we may choose so that $w_{i-1}\neq w_i$ for all $i=1,\dots,k$.

By the exchange condition for Coxeter groups, for all $i=1,\dots,k$ we either have $\ell(w_{i-1}s_{i}) > \ell(w_{i-1})$ or $\ell(s_{i}w_{i-1}) > \ell(w_{i-1})$. We claim that $x_{w_i} = x_{w_{i-1}}$ for all $i=1,\dots, k-1$.  If $\ell(w_{i-1} s_{i}) > \ell(w_{i-1})$, so that $x_{w_{i-1} s_{i}}=0$ by maximality of $\ell(w_{i-1})=\ell(w)$, comparing the $w_{i-1} s_{i}$-coefficients of $x T_{s_{i}} = T_{s_{i}}x$ using Lemma~\ref{lem:dual-action} yields
  \begin{gather*}
    (x T_{s_{i}})_{w_{i-1} s_{i}}
    = \rho'_{s_{i}} \delta_{w_{i-1} s_{i}} (x)
    = a_{s_{i}} x_{w_{i-1} s_{i}} + x_{w_{i-1}}
    = x_{w_{i-1}}
  \end{gather*}
 and
  \begin{gather*}
    (T_{s_{i}} x)_{w_{i-1} s_{i}}
    = \lambda'_{s_{i}} \delta_{w_{i-1} s_{i}} (x)
    = a_{s_{i}} x_{w_{i-1} s_{i}} + x_{s_{i} w_{i-1} s_{i}}
    = x_{w_{i}},
  \end{gather*}
  so that indeed $x_{w_i} = x_{w_{i-1}}$. The case $\ell(s_{i}w_{i-1}) > \ell(w_{i-1})$ is similar.

It follows that $x_{w_{k-1}} = x_w$. Since $\ell(w') > \ell(w_{k-1})$, so that $\ell(s_kw_{k-1}s_k)>\ell(w_{k-1}s_k)>\ell(w_{k-1})$, we compute as above using Lemma~\ref{lem:dual-action} that
  \begin{gather*}
    (x T_{s_k})_{w_{k-1} s_k}
    = \rho'_{s_{k}} \delta_{w_{k-1} s_{k}} (x)
    = a_{s_{k}} x_{w_{k-1} s_{k}} + x_{w_{k-1}}
    = x_{w_{k-1}}
    = x_w
  \end{gather*}
  and 
  \begin{gather*}
    (T_{s_k}x)_{w_{k-1} s_k}
    = \lambda'_{s_k} \delta_{w_{k-1} s_k}(x)
    = b_{s_k} x_{s_k w_{k-1} s_k}
    = 0
    \text{.}
  \end{gather*}
  Hence $x_w=0$, contradicting the choice of $w$.
\end{proof}

\bibliographystyle{amsalpha} 
\bibliography{these} 

\end{document}